\documentclass{gtmon_a}
\pdfoutput=1

\proceedingstitle{Exotic homology manifolds (Oberwolfach 2003)}
\conferencestart{29 June 2003}
\conferenceend{5 July 2003}
\conferencename{Workshop on Exotic Homology Manifolds}
\conferencelocation{Oberwolfach Mathematics Institute, Oberwolfach,
Germany}

\editor{Frank Quinn}
\givenname{Frank}
\surname{Quinn}

\editor{Andrew Ranicki}
\givenname{Andrew}
\surname{Ranicki}

\title[Path concordances as detectors of codimension-one manifold factors]
{Path concordances as detectors of codimension-one\\manifold factors}

\author{Robert J Daverman}
\givenname{Robert J}
\surname{Daverman}
\address{Department of Mathematics\\University of
Tennessee\\Knoxville\\\newline
Tennessee 37996-1300\\USA}
\email{daverman@math.utk.edu}
\urladdr{}

\author{Denise Halverson}
\givenname{Denise}
\surname{Halverson}
\address{Department of Mathematics\\Brigham Young
University\\\newline
Provo\\Utah\\84602\\USA}
\email{deniseh@math.byu.edu}
\urladdr{}

\volumenumber{9}
\issuenumber{}
\publicationyear{2006}
\papernumber{2}
\startpage{7}
\endpage{15}

\doi{}
\MR{}
\Zbl{}


\received{13 August 2003}
\revised{}
\accepted{13 August 2003}
\published{22 April 2006}
\publishedonline{22 April 2006}
\proposed{}
\seconded{}
\corresponding{}
\version{}

\arxivreference{}  

\makeatletter
\def\cnewtheorem#1[#2]#3{\newtheorem{#1}{#3}[section]
\expandafter\let\csname c@#1\endcsname\c@thm}

\newtheorem{thm}{Theorem}[section]
\cnewtheorem{lem}[thm]{Lemma}
\cnewtheorem{prop}[thm]{Proposition}
\cnewtheorem{cor}[thm]{Corollary}
\theoremstyle{definition}
\newtheorem*{defn}{Definition}

\makeatother  

\newcommand{\im}{\operatorname{im}}
\newcommand{\diam}{\operatorname{diam}}
\newcommand{\origin}{\operatorname{origin}}
\newcommand{\point}{\operatorname{point}}
\newcommand{\proj}{\operatorname{proj}}
\newcommand{\Map}{\operatorname{Map}}

\begin{document}

\begin{asciiabstract}
We present a new property, the Disjoint Path Concordances Property, of
an ENR homology manifold X which precisely characterizes when X times
R has the Disjoint Disks Property.  As a consequence, X times R is a
manifold if and only if X is resolvable and it possesses this
Disjoint Path Concordances Property.
\end{asciiabstract}

\begin{htmlabstract}
We present a new property, the Disjoint Path Concordances Property,
of an ENR homology manifold X which precisely characterizes when
X&times;<b>R</b> has the Disjoint Disks Property.  As a consequence,
X&times;<b>R</b> is a manifold if and only if X
is resolvable and it possesses this Disjoint Path Concordances Property.
\end{htmlabstract}

\begin{abstract} 
We present a new property, the Disjoint Path Concordances Property,
of an ENR homology manifold $X$ which precisely characterizes when
$X \times \mathbb{R}$ has the Disjoint Disks Property.  As a consequence,
$X \times \mathbb{R}$ is a manifold if and only if $X$
is resolvable and it possesses this Disjoint Path Concordances Property.
\end{abstract}

\maketitle


\section{Introduction}

Back in the 1950s R\,H Bing showed \cite{bing1,bing2} that his nonmanifold
``dogbone space'' is a Cartesian factor of Euclidean 4--space.  Since
then, topologists have sought to understand which spaces are factors of
manifolds. It is now known (due to the work of Edwards \cite{edw} and
Quinn \cite{Q}) that the manifold factors coincide with those
ENR homology manifolds which admit a cell-like resolution by a manifold;
equivalently, they are the ENR homology manifolds of trivial Quinn
index.  In particular, if $X$ has trivial Quinn index then
$X \times \mathbb{R}^k$ is a manifold for $k \geq 2$. Whether $X \times
\mathbb{R}$ itself is necessarily a manifold stands as a fundamental
unsettled question.

Several properties of a manifold factor $X$ of dimension $n$ assure that
its product with $\mathbb{R}$ is a manifold.  Among them are:
\begin{enumerate}
\item The singular (or nonmanifold) subset $S(X)$ of $X$ -- namely,
the complement of the maximal $n$--manifold contained in $X$ -- has
dimension at most $n{-}2$, where $n \ge 4$ (see the article by Cannon
\cite[Theorem 10.1]{C}).\label{prop1}
\item There exists a cell-like map $f\co M \to X$ defined on an
$n$--manifold such that $\dim \{x \in X : f^{-1}(x) \ne \point\} \le n-3$
(see the article by Daverman \cite[Theorem 3.3]{da81}).\label{prop2}
\item There exists a topologically embedded $(n{-}1)$--complex or an ENR
homology $(n{-}1)$--manifold $K$ with $S(X) \subset K \subset X$
(see the book by Daverman \cite[Corollaries
26.12A--12B]{dabook}).\label{prop3}
\item $X$ arises from a nested defining sequence, as defined by Cannon
and Daverman \cite{cd} \cite[Chapter 34]{dabook}, for the decomposition
into point inverses induced by a  cell-like map $f\co M \to X$.\label{prop4}
\item $X$ has the Disjoint Arc-Disk Property of
\cite[page 193]{dabook}.\label{prop5}
\end{enumerate}
Condition (\ref{prop5}) is implied by either (\ref{prop1}) or
(\ref{prop2}) but not by (\ref{prop3}) or (\ref{prop4}).

More pertinent to issues addressed in this manuscript,
Halverson \cite[Theorem 3.4]{dh02} proved that if an ENR homology $n$--manifold
$X$, $n \geq 4$, has a certain Disjoint Homotopies Property,
defined in the next section and abbreviated as DHP, then
$X \times \mathbb{R}$ has the more familiar Disjoint Disks Property,
henceforth abbreviated as DDP.
Because the DDP characterizes resolvable ENR homology
manifolds of dimension $n \geq 5$ as manifolds \cite{edw} \cite[Theorem
24.3]{dabook}, it
follows that $X \times \mathbb{R}$ is a genuine manifold if
$X$ is resolvable and has this DHP. Still
unknown are both whether all such ENR homology manifolds have DHP
and whether $X$ having said DHP is a necessary condition for $X \times
\mathbb{R}$ to be a manifold.

Since an ENR homology $n$--manifold $X, n \ge 4,$ has DHP if it
satisfies any of the properties mentioned in the second paragraph
except (\ref{prop3}), the sort of homology manifolds that might fail
to have it are the ghastly examples of Daverman and Walsh \cite{dw}.
Halverson \cite{dh03} has constructed some related ghastly examples that
do have DHP.

In hopes of better understanding
the special codimension one manifold factors, this paper builds on Halverson's
earlier work to present a necessary and sufficient condition,
the Disjoint Path Concordances Property defined at the outset of Section 2,
for $X \times
\mathbb{R}$ to be a manifold (again provided $\dim X \geq 4$).

\section{Preliminaries}

Throughout what follows both $D$ and $I$ stand for the unit interval,
[0,1]. A metric space $X$ is said to have the \emph{Disjoint
Homotopies Property} if any pair of path homotopies $f_1,f_2\co D \times
I \to X$ can be approximated, arbitrarily closely, by
homotopies $g_1,g_2\co D\times I \to X$ such that
$$g_1(D \times t) \cap g_2(D \times t) = \emptyset, \quad \text{for all } t \in
I.$$

\begin{defn}
A \emph{path concordance} in a space $X$ is a map $F\co D \times I \to
X \times I$ such that $F(D \times e) \subset X \times e, e \in
\{0,1\}.$
\end{defn}

Let $\proj_X\co X \times I \to X$ denote projection.

\begin{defn}
A metric space $(X,\rho)$ satisfies the \emph{Disjoint Path
Concordances Property (DCP)} if, for any two path homotopies
$f_1,f_2\co D \times I \to X$ and any $\epsilon > 0$, there
exist path concordances $F_1,F_2\co  D \times I \to X \times I$ such that
$$F_1(D \times I) \cap F_2(D \times I) = \emptyset$$
and $\rho (f_i, \proj_X F_i) < \epsilon$.
\end{defn}

A \emph{homology $n$--manifold} $X$ is a locally compact metric space such that
$$H_*(X,X-\{x\}; \mathbb{Z}) \cong H_*(\mathbb{R}^n,\mathbb{R}^n -
\{\origin\}; \mathbb{Z})$$
for all $x \in X$.
An \emph{ENR homology $n$--manifold} $X$ is an homology $n$--manifold
which is homeomorphic to a retract of an open subset of some Euclidean
space (ENR is the abbreviation for ``Euclidean neighborhood retract'');
equivalently, $X$ is a homology $n$--manifold which is both
finite-dimensional and locally contractible.

A homology $n$--manifold $X$ is \emph{resolvable} provided there
exists a surjective, cell-like mapping $f\co M \to X$ defined on an
$n$--manifold $M$. Quinn \cite{Q} has shown that (connected) ENR
homology $n$--manifolds $X$, for $n \geq 4$, are resolvable if and only
if a certain index $i(X) \in 1 + 8\mathbb{Z}$ equals 1.

A metric space $X$ has the \emph{Disjoint Arcs Property (DAP)} if every pair
$f_1,f_2\co D \to X$ of maps can be approximated, arbitrarily closely,
by maps $g_1,g_2\co D \to X$ for which $g_1(D) \cap g_2(D) = \emptyset$.
All ENR homology $n$--manifolds, for $n \ge 3$, have the DAP.

The following Homotopy Extension Theorem is fairly standard. We state it
here because it will be applied several times in our arguments.

\begin{thm}[Controlled Homotopy Extension Theorem (CHET)]\label{HET}
Suppose $X$ is a metric ANR, $C$ is a compact subset of $X$,
$j\co C \to X$ is the inclusion map, and
$\epsilon > 0$. Then there exists $\delta > 0$ such that for
each map $f\co Y \to C$ defined on a normal space $Y$,
each closed subset $Z$ of $Y$,  and each map
$g_Z\co Z \to X$ which is $\delta$--close to  $jf|_Z$,
$g_Z$ extends to a map $g\co Y
\to X$ which is $\epsilon$--homotopic to $jf$. In particular,
for any open set $U$ for which $Z \subset U \subset Y$, there is a
homotopy $H\co Y \times I \to X$ such that
\bgroup\begin{enumerate}
    \item $H_0 = jf$ and $H_1 = g$
    \item $g|_Z = g_Z$
    \item $H_t|_{Y-U} = jf|_{Y-U}$ for all $t \in I$
    \item $\diam(H(y \times I)) < \epsilon$ for all $y \in Y$.
\end{enumerate}\egroup
\end{thm}

\section{Main results}

In this section we will demonstrate that DCP characterizes
codimension-one manifold factors among ENR homology $n$--manifolds, for
$n \ge 4$,  of trivial Quinn index.
Essentially the DCP condition requires that any pair of level preserving path
homotopies into $X\times I$ can be ``approximated'' by disjoint
path concordances, where the approximation is measured in the $X$
factor. The following crucial proposition demonstrates that the DCP
condition implies that any pair of level preserving path
homotopies can be approximated, as measured in $X \times I$,
by disjoint path concordances.

\begin{prop}\label{prop-two}
Suppose that $(X, \rho)$ is a metric ANR with DAP.  Then $X$ has DCP if and only
if given any pair of level preserving maps $f_1,f_2\co D \times I \to X
\times I$ and $\epsilon > 0$ there are maps $g_1,g_2\co D
\times I \to X \times I$ such that
\begin{enumerate}
\item $f_i$ and $g_i$ are $\epsilon$--close in $X \times I$,
\item $g_i|_{D \times \partial I}$ is level preserving, and
\item $g_1(D \times I) \cap g_2(D \times I) = \emptyset$.
\end{enumerate}
Moreover, if $f_1(D \times \partial I) \cap f_2(D \times \partial
I) = \emptyset$, then we also may require that $g_i|_{D \times \partial
I} = f_i|_{D \times \partial I}$.
\end{prop}

\begin{proof}
Assume $X$ has DCP. Use $\widetilde{\rho}$ to denote the obvious sum   metric on $X \times I$.
Let $f_1,f_2\co D \times I \to X \times I$ be level-preserving
maps, and let $\epsilon
> 0$. Choose
$$0=t_0 < t_1 < \cdots < t_m=1$$
such that $t_k-t_{k-1}
< \epsilon/2$ for $k=1,\ldots,m$.
Define
$$J_k =[t_k, t_{k-1}],\qquad
f_i[k] = f_i|_{D \times J_k}.$$
Since $X$ has DAP then
applying \fullref{HET} (near $\proj_X(f_1(D \times I) \cup
f_2(D \times I)))$
we may assume that $f_1(D \times t_k) \cap
f_2(D \times t_k) = \emptyset$. Choose $\eta > 0$ so that
$$\widetilde{\rho}(f_1(D \times t_k), f_2(D \times t_k)) > \eta$$
for
$k=1,\ldots,n$. Let $\delta
> 0$ satisfy CHET for $X \times I$, for a small compact neighborhood $C$ of
$f_1(D \times I) \cup f_2(D \times I)$ and
$\min\{\eta/2,\epsilon/2\}$. Then by DCP,
applied to the intervals $J_k$, there are maps
$g_i[k]\co D\times J_k \to X\times J_k$ satisfying
\begin{enumerate}
\item $g_i[k](D \times J_k) \subset C$,
\item $\rho (\proj_X f_i[k], \proj_X g_i[k]) < \delta$,
\item $g_i[k]|_{D \times \partial J_k}$ is level preserving, and
\item $g_1[k](D \times J_k) \cap g_2[k](D \times J_k) = \emptyset$.
\end{enumerate}
By CHET and the choice of $\delta$ there
are maps $G_i[k]\co D\times J_k \to X\times J_k$ satisfying
\begin{enumerate}
\item $\rho (\proj_X f_i[k], \proj_X G_i[k]) < \epsilon/2$,
\item $G_i[k]|_{D \times \partial J_k} = f_i[k]|_{D \times \partial J_k}$, and
\item $G_1[k](D \times J_k) \cap G_2[k](D \times J_k) = \emptyset$.
\end{enumerate}
Set $G_i = \bigcup_k G_i[k]$. Confirming that $G_1$ and $G_2$ are the desired
maps is straightforward.

The reverse direction is trivial.  It merely requires treating any pair of
path homotopies $D \times I \to X$ as level-preserving maps
$D \times I \to X \times I.$
\end{proof}

\begin{defn}
A metric space ($X,\rho$) satisfies the \emph{Disjoint 1--Complex
Concordances Property (DCP*)} if, for any two  homotopies $f_i\co K_i
\times I \to X$, for $i=1,2$, where $K_i$ is a finite 1--complex, and
any $\epsilon > 0$ there exist concordances $F_i\co  K_i \times I \to
X \times I$ such that
$$F_1(K_1 \times I) \cap F_2(K_2 \times I) = \emptyset,$$
$F_i(K_i \times e) \subset X \times e$ for $e \in \partial I$, and $\rho
(f_i, \proj_X F_i) < \epsilon$.
\end{defn}

\begin{prop}
Suppose $X$ is a locally compact, metrizeable ANR with DAP.  Then $X$ has DCP
if and only if $X$ has DCP*. \label{prop-three}
\end{prop}

\begin{proof}
This argument is similar to the one showing the equivalence of the
DDP with approximability of maps defined on finite 2-complexes by
embedding \cite[Theorem 24.1]{dabook}, and also to another one showing
the equivalence of DHP (for paths) and a Disjoint Homotopies Property
for 1-complexes \cite[Theorem 2.9]{dh02}.  We supply the short proof
for completeness.

To show the forward direction, endow $X$ with a complete metric $\rho$.
Let $K_1$ and $K_2$ be finite 1--simplicial complexes, and define
$$\mathcal{H} = \bigl\{ (f_1,f_2) \in \Map(K_1 \times I, X \times I) \times
  \Map(K_2 \times I, X \times I) : f_i|_{D \times \partial I}
\text{ level-preserving} \bigr\}$$
with the uniform metric. Note that $\mathcal{H}$ is a complete
metric space and, therefore, a Baire space. For $\sigma_i \in
K_i$ , let
$$\mathcal{O}(\sigma_1 , \sigma_2 )=\bigl\{ (f_1,f_2) \in \mathcal{H} :
  f_1(\sigma_1 \times I) \cap f_2(\sigma_2 \times I) = \emptyset \bigr\}.$$
Clearly $\mathcal{O}(\sigma_1 , \sigma_2 )$ is open in
$\mathcal{H}$. To see that $\mathcal{O}(\sigma_1 , \sigma_2 )$ is
dense in $\mathcal{H}$, let $\epsilon > 0$. Choose $\delta > 0$ to
satisfy CHET for $X \times I$, for a small compact neighborhood
of $\bigcup_i f_i(\sigma_i \times I)$, and $\epsilon$. Then choose
$\eta > 0$ to satisfy CHET for $X, \proj_X \bigl(\bigcup_i f_i(\sigma_i
\times I)\bigr)$ and $\delta$. By \fullref{prop-two} there are
$\eta$--approximations $g_i$ to $f_i|_{\sigma_i \times I}$ that are
disjoint path concordances. First apply CHET to extend $g_i$ over
$(\sigma_i \times I) \cup (K_i \times \partial I)$ so that the new
$g_i$ is level preserving on $K_i \times \partial I$ and
$\delta$--close to $f_i|_{(\sigma_i \times I) \cup (K_i \times
\partial I)}$. Then apply CHET again to extend $g_i$ over $K_i
\times I$ so that $g_i$ is now $\epsilon$--close to $f_i$. Thus,
$\mathcal{O}(\sigma_1 , \sigma_2 )$ is dense in $\mathcal{H}$.
Since $\mathcal{H}$ is a Baire space,
$$\mathcal{O} =
  \underset{(\sigma_1,\sigma_2) \in K_1 \times K_2} \bigcap
  \mathcal{O}(\sigma_1, \sigma_2)$$
is dense in $\mathcal{H}$. Note
that if $(f_1,f_2) \in \mathcal{O}$ then $f_1(K_1 \times I) \cap
f_2(K_2\times I) = \emptyset$. Hence, $X$ has DCP*.

The other direction is trivial.
\end{proof}

\begin{defn}  A \emph{topography $\Upsilon$ on $D \times I$} consists
of the following elements:
\begin{enumerate}
\item A set $\{ J_1, \ldots ,J_m \}$ of consecutive intervals in
  $\mathbb{R}$ such that $J_j = [t_{j-1},t_j]$
  where $t_0 < t_1 < \cdots < t_m$.
\item A finite set $\{ L_0, \ldots ,L_m \}$ of complexes embedded
  in $D \times I$ of dimension at most one. These complexes are called the
  \emph{transition levels}.
\item A finite set $\{ K_1, \ldots ,K_m \}$ of 1--complexes.  These
  complexes are called the \emph{level factors}.
\item A set $\bigl\{ \phi_j\co K_j \times J_j \to
  D \times I \bigr\}_{j=1,\ldots,m}$ of maps
  which satisfy the following conditions:
  \begin{enumerate}
  \item $L_0 = \phi_1(K_1 \times \{ t_0 \})$
  \item $L_m = \phi_m(K_m \times \{ t_m \})$
  \item $L_j = \phi_j(K_j \times \{ t_j \}) \cup \phi_{j+1}
    (K_{j+1} \times \{ t_j \})$ for $j=1,\ldots,m{-}1$
  \item $\phi_j|_{K_j \times \text{int }J_j}$ is an
    embedding for each $j=1,\ldots,m$
  \item $\bigcup_{j=1}^m \text{im } (\phi_j) = D \times I$
  \end{enumerate}
\end{enumerate}
\end{defn}

It is shown by Halverson \cite{dh02} that a p.l. general position
approximation of the projection of a map $f\co D \times I \to X
\times \mathbb{R}$ to the $\mathbb{R}$ factor induces a
topographical structure on the domain $D \times I$ of $f$.

\begin{defn}
A map $f\co D \times I \to X \times \mathbb{R}$ is a
\emph{topographical map} if there is a topography, $\Upsilon$, on
$D \times I$ which is level preserving in the sense that for each
map $\phi_j\co K_j \times J_j \to D \times I$ of the topography, $f
\circ \phi_j(K_j \times t) \subset X \times t$ for all $t \in
J_j$.
\end{defn}

Note that a topographical structure on $D \times I$ is in no way related
to the product structure of $D \times I$.  The proof of the following
lemma is provided in \cite[Theorem 3.3]{dh02}.

\begin{lem}
Every map $f\co D \times I \to X \times \mathbb{R}$ can be
approximated by a topographical map $g\co D \times I \to X \times
\mathbb{R}$. \label{TOP}
\end{lem}

\begin{thm}[Disjoint Concordances Theorem]
\label{thm-dct}
Suppose $X$ is a locally compact, metric ANR with DAP.  Then $X$ has DCP
if and only if $X \times \mathbb{R}$ has DDP.
\end{thm}

\begin{proof}
($\Leftarrow$)  Given two path homotopies $f_1,f_2\co D\times I
\to X$,
treat them as level-preserving maps $f_1,f_2\co D\times I
\to X \times I$. Applying DAP, and using CHET as before, we may assume
without loss of generality that
$f_1(D \times \partial I) \cap f_2(D \times \partial I) = \emptyset$.
Since $X \times (0,1)$ has DDP,
each $f_i$ can be approximated, fixing the actions on
$D \times \{0,1\}$, by an $\epsilon$--close map $g_i$, such that
the images of $g_1,g_2$ are disjoint.  The DCP follows.

($\Rightarrow$) Given maps $f_1,f_2\co I^2 = D \times I \to X \times
\mathbb{R}$, by \fullref{TOP} we may assume that each $f_i$ is a
topographical map with topography $\Upsilon[i]$. An object $O$ in the
definition of $\Upsilon[i]$ will be denoted as $O[i]$. Note that we also
may assume the following:
\begin{enumerate}
\item The set of intervals $\{ J_j[i] \}$ are the same for $i=1,2$.
This follows from subdividing the intervals appropriately as
outlined in \cite[Theorem 3.4]{dh02}.
\item $f_1(\bigcup L_j[1]) \cap f_2(\bigcup L_j[2]) = \emptyset$.
This follows from DAP and CHET.
\end{enumerate}
By DCP* there exist maps $\psi_j[i]\co K_j[i] \times J_j$, approximating
$\phi_j f_i$, such that
$$\im\psi_j[1] \cap \im\psi_j[2] = \emptyset.$$
Applying \fullref{HET}, we may assume that
$$\psi_j[i]|_{K_j[i] \times \partial J_j}=\phi_j
f_i|_{K_j[i] \times \partial J_j}.$$
(Recall that $\phi_j
f_i(K_j[i] \times \partial J_j) \subset (L_j \cup L_{j-1})$.)  Set
$g_i = \bigcup_j \psi_j[i]$. Then $g_1$ and $g_2$ are
the desired disjoint approximations of $f_1$ and $f_2$.
\end{proof}

\begin{cor}
\label{cor-three-five}
An ENR homology $n$--manifold $X$, for $n \ge 4$, has DCP
if and only if $X \times \mathbb{R}$ has DDP.
\end{cor}

When $n=3$, \fullref{cor-three-five} is formally but vacuously true:
since no ENR homology 4--manifold has DDP,
no such homology 3--manifold can have DCP.

\begin{cor}
Let $X$ be an ENR homology $n$--manifold, $n \ge 4$.  Then
$X \times \mathbb{R}$ is a manifold
if and only if $X$ has trivial Quinn index and
satisfies DCP.
\end{cor}
\fullref{thm-dct} demonstrates the equivalence in $X \times
\mathbb{R}$ of DCP and DDP.
By the standard combination of the work of Edwards \cite{edw} and of
Quinn \cite{Q}, $X \times \mathbb{R}$ is a manifold if and only if it
has both this latter property and trivial Quinn index.  Furthermore,
by \cite{Q}, $i(X \times \mathbb{R}) =  i(X)$, without regard to any of
these general position properties.

\section{Questions}
(1)\qua Is every finite-dimensional Busemann space (see
Thurston \cite{pt}) $X$ necessarily a manifold?
Actually, there are two unsettled questions here: is $i(X)$, the Quinn index, trivial?
When $\dim X > 4$, must $X$ have DDP? It may be of interest to add that
$X$ is known to be a manifold if $\dim X \le 4$ \cite{pt}.

The same pair of concerns crops up in the following setting.

(2)\qua  Suppose, for any two points $p,q$ of the compact,
finite-dimensional metric space $X$, there is a homeomorphism
from the suspension of a space $Y$ onto $X$ that carries the suspension
points onto $p,q$.  Is $X$ a manifold?

(3)\qua  Given an ENR homology $n$--manifold $X$, for $n \ge 4$, can maps
$f,g\co I^2 \to X$ be approximated by $F,G\co I^2 \to X$ for which there
exists a 0--dimensional $F_\sigma$ set $T \subset I^2$ such that
$$F(I^2 - T) \cap G(I^2 - T) = \emptyset?$$
If so, $X \times \mathbb{R}$ will satisfy the DDP.  Recall that $X
\times \mathbb{R}$ satisfies the DDP when $S(X)$, the singular set
of $X$, is at most $(n{-}2)$--dimensional. A key to the argument is
that then one can obtain $F,G\co I^2 \to X$ and a compact,
0--dimensional $T$ with $F(I^2 - T) \cap G(I^2 - T) = \emptyset$.
Similar reasoning applies with $F_\sigma$--subsets $T$ in place of
compact subsets.  It may be worth noting that, because $X$ does
satisfy the DAP, one can easily obtain (maps $F,G$ and) such a $T$
which is a 0--dimensional $G_\delta$--subset of $I^2,$ but the
argument requires a 0--dimensional $F_\sigma$--subset, which is more
``meager''.


\bibliographystyle{gtart}
\bibliography{link}

\end{document}